\documentclass[a4paper, leqno,11pt]{amsart}

\usepackage{amssymb, amsmath, amsthm,relsize}
\usepackage{xargs} 
\usepackage{hyperref} 
\usepackage[colorinlistoftodos,textsize=small]{todonotes} 
\usepackage{subcaption,graphicx,xcolor}
\usepackage{forest}

\numberwithin{equation}{section}
\newtheorem{thm}{Theorem}

\newtheorem{corollary}[thm]{Corollary}
\newtheorem{cor}[thm]{Corollary}
\newtheorem{lemma}[thm]{Lemma}
\newtheorem{lem}[thm]{Lemma}

\theoremstyle{remark}

\newtheorem*{theorem*}{Theorem}

\newtheorem{definition}{Definition}
 
\newcommand{\norm}[1]{\left\Vert#1\right\Vert}
\newcommand{\brkt}[1]{\left(#1\right)}

\newcommand{\abs}[1]{\left|#1\right|}

\newcommand{\ip}[1]{\left\langle#1\right\rangle} 

         \newcommand{\B}{\mathbb{B}}

	\newcommand{\ch}{\mathcal{H}}

	\newcommand{\ipf}[1]{\ip{#1}_f}
	\newcommand{\normf}[1]{\norm{#1}_f}
	\newcommand{\rtt}{\frac{\sqrt{2}}{2}}
	\newcommand{\half}{\frac{1}{2}}
 
	\newcommandx{\concern}[2][1=]{\todo[color = red!70!,#1]{Concern: #2}} 
	\newcommandx{\refq}[2][1=]{\todo[color = yellow!40!,#1,]{Reference: #2}} 
	\newcommandx{\wording}[2][1=]{\todo[color = violet!50!,#1,]{Wording: #2}} 
	\newcommandx{\alan}[2][1=]{\todo[color = green!25!,#1]{#2}}
	\newcommandx{\mere}[2][1=]{\todo[color = blue!25!,#1]{#2}}

        \newcommand{\McC}{\raise.5ex\hbox{c}}

\title[Optimal approximants and orthogonal polynomials II]{Optimal approximants and orthogonal polynomials in several variables II: \\families of polynomials in the unit ball}

\author[Sargent]{Meredith Sargent}
\address{Department of Mathematics, University of Arkansas, Fayetteville, AR 72701, U.S.A.}
\email{sargent@uark.edu}

\author[Sola]{Alan Sola}
\address{Department of Mathematics, Stockholm University, 106 91 Stockholm, Sweden}
\email{sola@math.su.se}
\date{\today}


\begin{document}

\begin{abstract} 
We obtain closed expressions for weighted orthogonal polynomials and optimal approximants associated with the function $f(z)=1-\frac{1}{\sqrt{2}}(z_1+z_2)$ and a scale of Hilbert function spaces in the unit $2$-ball having reproducing kernel $(1-\langle z,w\rangle)^{-\gamma}$, $\gamma>0$. Our arguments are elementary but do not rely on reduction to the one-dimensional case.
 \end{abstract}
\maketitle


%

\section{Introduction}
This note continues recent work in \cite{MSAS1} concerning certain families of polynomials connected with approximation in spaces of analytic functions, and orthogonal polynomials in weighted spaces. In the paper \cite{MSAS1}, we discussed the notion of {\it optimal approximants} to $1/f$ for a holomorphic function $f$ belonging to a Hilbert function space in $\mathbb{C}^n$, and pointed out connections with orthogonal polynomials in certain weighted spaces, with weight determined by the same target function $f$. We presented some elementary examples of optimal approximants and orthogonal polynomials in several variables, and to obtain concrete closed-form representations of these objects, we relied on one-variable results together with suitable transformations. 

In this note, we present a further family of examples of weighted orthogonal polynomials and optimal approximants in several variables. We use a direct, elementary approach to go beyond cases that admit easy reduction to essentially one-variable problems. For simplicity, we focus on two variables, the target function $f=1-\frac{1}{\sqrt{2}}(z_1+z_2)$, and a scale of spaces of functions in the unit ball $\mathbb{B}^2=\{(z_1,z_2)\in\mathbb{C}^2\colon |z_1|^2+|z_2|^2<1\}$,
but some of our arguments potentially extend to higher dimensions, at the price of more cumbersome notation and more involved proofs.

We consider a scale of reproducing kernel Hilbert spaces that have recently featured in work of Richter and Sunkes \cite{RicSun16}. For further background on this kind of spaces, see for instance \cite{Zhu,CSW11,CHZ18} and the references therein. Fix $\gamma>0$ and let $\ch_\gamma$ denote the reproducing kernel Hilbert space in $\mathbb{B}^d$ associated with the reproducing kernel 
\[k_{\gamma}(z;w)=\frac{1}{\brkt{1-\ip{z,w}}^\gamma }, \quad z,w\in\mathbb{B}^d.\] 
The $\mathcal{H}_{\gamma}$ include well-known spaces such as the {\it Drury-Arveson space} ($H^2_d=\ch_1$), the {\it Hardy space} of $\B^2$ ($H^2(\partial \B_d)=\ch_d$), and the {\it Bergman space} of the $2$-ball ($A^2(\B_d)=\ch_{d+1}$).  In two variables, the norm in $\mathcal{H}_{\gamma}$ of an analytic function $f=\sum_{m=0}^\infty \sum_{n=0}^\infty \hat{f}(m,n)z_1^m z_2^n$
can be expressed as
		\begin{equation}
			\norm{f}^2_\gamma =\sum_{m=0}^\infty \sum_{n=0}^\infty a_{m,n}\abs{\hat{f}(m,n)}^2,
		\end{equation}
	where
		\begin{equation}
			a_{m,n}=
			\begin{cases}
				1 & m=n=0,\\
				\frac{m!n!}{(\gamma+m+n-1)\cdots(\gamma+1)\cdot \gamma} & \text{otherwise.}
			\end{cases}
		\end{equation}
We observe that polynomials are dense in all the $\mathcal{H}_{\gamma}$, monomials are orthogonal, and multiplication by the coordinate functions furnish bounded linear operators.		

We now state the definition of optimal approximants; see \cite{Chui80,SecoSurvey,JentZeros19,BetalPrep2,MSAS1} for more comprehensive discussions and references. 
Enumerating the monomials in two variables in some fixed way, we write $\chi_j$ for the $j$th monomial in this ordering, and set 
$\mathcal{P}_n=\mathrm{span}\{\chi_j\colon j=0,\ldots, n\}$. In this note, we work with {\it degree lexicographic order}. Monomials are ordered by increasing total degree, and ties between two monomials of the same total degree are broken lexicographically. See \cite{GWsiam07,GeroJEMS14} and the references therein for background material. Explicitly, we have \[1\prec z_1 \prec z_2 \prec z_1^2\prec  z_1z_2\prec z_2^2\prec z_1^3\prec z_1^2z_2\prec \cdots\,\, ,\] 
so that $\chi_4=z_1z_2$, $\chi_5=z_2^2$, and so on. For pairs of natural numbers $(j,k)$ and $(m,n)$, we will take $(j,k)\prec (m,n)$ to signify that $z_1^jz_2^k\prec z_1^mz_2^n$. 

\begin{definition}[Optimal approximants]
Let $f\in \mathcal{H}_{\gamma}$ be given. We define the $n$th order {\it optimal approximant}  to $1/f$ in $\mathcal{H}_{\gamma}$, relative to $\mathcal{P}_n$, as
$p_n^{\ast}=\mathrm{Proj}_{f\cdot \mathcal{P}_n}[1]/f$,
where $\mathrm{Proj}_{f\cdot \mathcal{P}_n}\colon \mathcal{H}_{\gamma}\to f\cdot \mathcal{P}_n$ is the orthogonal projection onto the closed subspace $f\cdot \mathcal{P}_n \subset \mathcal{H}_{n}$.
\end{definition}
Given some $f\in \mathcal{H}_{\gamma}$, optimal approximants can be viewed as polynomial substitutes for the function $1/f$, the point being that $1/f$ may fall outside of $\mathcal{H}_{\gamma}$. Optimal approximants arise in several contexts, for instance cyclicity problems and filtering theory, see \cite{SecoSurvey,MSAS1}. The papers \cite{FMS14,JAM15,BetalPrep2} discuss some methods for computing optimal approximants, but closed formulas are only known in a few instances. Multi-variable examples have so far only been obtained as a consequence of one-variable results.

\begin{definition}[Weighted orthogonal polynomials]
Let $f\in \mathcal{H}_{\gamma}$ be fixed. We say that a sequence $\{\phi_j\}_{j \in \mathbb{N}} \subset \mathbb{C}[z_1,z_2]$ consists of {\it weighted orthogonal polynomials} with respect to $f$ if 
$\{\phi_j\}$ is an orthogonal basis for the Hilbert space $\mathcal{H}_{\gamma,f}$ with inner product given by $\langle g,h\rangle_{\gamma,f}
\colon =\langle gf,hf \rangle_{\mathcal{H}_{\gamma}}$.
\end{definition}
There is an important connection between optimal approximants and orthogonal polynomials, as is explained in \cite{JLMS16,MSAS1}. Namely, if $\{p_n^*\}$ denote the optimal approximants to $1/f$, $f\in \mathcal{H}_{\gamma}$, and $\{\phi_n\}$ are orthogonal polynomials in the weighted space $\mathcal{H}_{\gamma,f}$, respectively, then
\begin{equation}
p_n^*(z)=\sum_{k=0}^n\langle 1, f\psi_k\rangle_{\mathcal{H}_{\gamma}} \psi_k(z),
\label{OAvsOGformula}
\end{equation}
where $\psi_k=\phi_k/\|\phi_k\|_{\gamma,f}$. This means that if we determine $\{\phi_k\}_k$ explicitly for some given weight $f$, then we also obtain formulas for the optimal approximants to $1/f$. Implementing this strategy in practice in $\mathcal{H}_{\gamma}$ and for the function $f=1-\frac{1}{\sqrt{2}}(z_1+z_2)$ is the goal of this note.

\section{A family of orthogonal polynomials}
We begin with an elementary lemma.
	\begin{lem}\label{lem:winner_monoms}
		Let $f(z_1,z_2)=1-a(z_1+z_2)$ and let $\mathcal{H}$ be a reproducing kernel Hilbert space in which the monomials are orthogonal. Consider $\mathcal{H}_f$, the space weighted by $f$ with inner product $\langle g, h\rangle_{\mathcal{H}_f} :=\langle gf , hf\rangle_{\mathcal{H}}$. For nonnegative integers $j,k,m,n$, we have
				\begin{multline*}
					\ip{z_1^j z_2^k,\,z_1^m z_2^n}_f=\\
						\begin{cases}
							\norm{z_1^j z_2^k}^2 + a^2 \norm{z_1^{j+1}z_2^k}^2 +a^2 \norm{z_1^j z_2^{k+1}}^2 &\quad 
									\mathrm{if}\;\;\parbox{3cm}{$m=j$, $n=k$,}\\	~&~\\
							-a\norm{z_1^j z_2^k}^2 &\quad 
									\mathrm{if}\;\;\parbox{3cm}{$m=j-1$, $n=k$, or \\$m=j$, $n=k-1,$}\\	~&~\\
							-a\norm{z_1^{j+1} z_2^k}^2 &\quad 
									\mathrm{if}\;\;\parbox{3cm}{$m=j+1$, $n=k$,}\\	~&~\\
							-a\norm{z_1^{j} z_2^{k+1}}^2 &\quad 
									\mathrm{if}\;\;\parbox{3cm}{$m=j$, $n=k+1$,}\\	~&~\\
							a^2\norm{z_1^{j+1} z_2^{k}}^2 &\quad 
									\mathrm{if}\;\;\parbox{3cm}{$m=j+1$,\\ $n=k-1$,}\\	~&~\\
							a^2\norm{z_1^{j} z_2^{k+1}}^2 &\quad 
									\mathrm{if}\;\;\parbox{3cm}{$m=j-1$,\\ $n=k+1$,}\\	~&~\\
							0 &\quad \mathrm{otherwise.}
						\end{cases}
				\end{multline*}
	\end{lem}
\begin{proof}
This amounts to expanding the inner product and reading off terms.
\end{proof}
Recall the standard definition of the {\it Pochhammer symbol} for $\gamma$ real:
\[(\gamma)_n=\gamma\cdot (\gamma+1)\cdots (\gamma+n-1), \quad n\geq 0.\]
	\begin{thm}\label{thm:closed_form}
		In $\ch_\gamma$, weighted by $f(z_1,z_2)=1-\frac{\sqrt{2}}{{2}}\brkt{z_1+z_2}$, let $\phi_{j,k}$ be the first orthogonal polynomial containing $z_1^j z_2^k$ (with respect to degree lexicographic order). Then $\phi_{j,k}$ has the form
		\begin{equation}\label{eq:whichterms}
			\phi_{j,k}(z_1,z_2) = \sum_{m=0}^j \sum_{n=0}^k \hat\phi_{j,k}(m,n) z_1^m z_2^n
		\end{equation}
		where the coefficients $\hat\phi_{j,k}(m,n)$ are given by
		\begin{equation}\label{eq:closedcoeffs}
			\hat\phi_{j,k}(m,n) = \brkt{\frac{\sqrt{2}}{2}}^{j+k-m-n} 
														\frac{(\gamma)_{m+n+1}}{(\gamma)_{j+k+1}} 
														\brkt{ \frac{j!k!}{m!n!} \cdot\frac{\brkt{j+k-m-n}!}{\brkt{j-m}!\brkt{k-n}!} }.
		\end{equation}
		Moreover,
		\begin{align}
			\norm{\phi_{j,k}}^2_f &= \frac{\gamma+j+k+1}{\gamma+j+k}\cdot \frac{j!k!}{(\gamma)_{j+k}}.
			\label{eq:closednorms}
		\end{align}
	\end{thm}
	\begin{proof}
		We shall prove this using strong induction. First, $\phi_{0,0}(z_1,z_2)=1$, and by Lemma \ref{lem:winner_monoms}, 
		\begin{align*}
			\norm{\phi_{0,0}}^2_f 
				= \norm{1}^2_f 
				&= \norm{1}^2 + \brkt{\frac{\sqrt{2}}{2}}^2\norm{z}^2 + \brkt{\frac{\sqrt{2}}{2}}^2\norm{z}^2\\
				&= 1 + \frac{1}{2\gamma} + \frac{1}{2\gamma}= \frac{\gamma+1}{\gamma}
		\end{align*} as needed.
		Now consider $\phi_{j,k}$ and assume that for all $(J,K)\prec(j,k)$, the polynomial $\phi_{J,K}$ has the desired form, coefficients, and norm. Using the Gram-Schmidt algorithm,
		\begin{equation}\label{eq:gramschmidt}
			\phi_{j,k}(z_1,z_2)= z_1^jz_2^k - \sum_{(J,K)\prec(j,k)} \frac{ \ipf{z_1^jz_2^k,\,\phi_{J,K}} }{\norm{\phi_{J,K}}_f^2}\phi_{J,K}.
		\end{equation}
		Each $\phi_{J,K}$ has the form \eqref{eq:whichterms}, and by Lemma \ref{lem:winner_monoms},  we see that there are only three $\phi_{J,K}$ with $(J,K)\prec(j,k)$ that give a non-zero inner product: $\phi_{j,k-1}$, $\phi_{j-1,k}$, and $\phi_{j+1,k-1}$. Noting that $\hat\phi_{J,K}(J,K)=1$ and applying Lemma \ref{lem:winner_monoms} gives that
		\begin{align}
			\ipf{z_1^jz_2^k,\,\phi_{j,k-1}} &=\ipf{z_1^jz_2^k,\,z_1^jz_2^{k-1}} = -\rtt \frac{j!k!}{(\gamma+j+k-1)\cdots(\gamma+1)\cdot \gamma}\label{eq:iplessk}\\
			\ipf{z_1^jz_2^k,\,\phi_{j-1,k}} &=\ipf{z_1^jz_2^k,\,z_1^{j-1}z_2^{k}} = -\rtt \frac{j!k!}{(\gamma+j+k-1)\cdots(\gamma+1)\cdot \gamma}\label{eq:iplessj}\\
			\ipf{z_1^jz_2^k,\,\phi_{j+1,k-1}}&=\ipf{z_1^jz_2^k,\, z_1^{j+1}z_2^{k-1} + \hat\phi_{j+1,k-1}(j,k-1)z_1^jz_2^{k-1}}\label{eq:cancels}\\
					&=\ipf{z_1^jz_2^k,\,z_1^{j+1}z_2^{k-1}}\nonumber\\
					&\qquad+ \hat\phi_{j+1,k-1}(j,k-1)\ipf{z_1^jz_2^k,\, z_1^jz_2^{k-1}}.\nonumber
		\end{align}
		The right hand side of \eqref{eq:cancels} is equal to zero: by Lemma \ref{lem:winner_monoms}, 
		\begin{equation}
			\ipf{z_1^jz_2^k,z_1^{j+1}z_2^{k-1}} = \frac{1}{2}\frac{(j+1)!k!}{(\gamma+j+1+k-1)\cdots(\gamma+1)\cdot \gamma},
		\end{equation}
		and by the inductive hypothesis about the norm of $\phi_{j+1,k-1}$ and Lemma \ref{lem:winner_monoms},
		\begin{align}
			\hat\phi_{j+1,k-1}(j,k-1)\ipf{z_1^jz_2^k,z_1^jz_2^{k-1}} 
				&= \rtt \frac{j+1}{\gamma+j+k}\cdot \brkt{-\rtt}\frac{j!k!}{(\gamma+j+k-1)\cdots(\gamma+1)\cdot \gamma}\nonumber\\
				&= -\frac{1}{2}\frac{(j+1)!k!}{(\gamma+j+k)\cdots(\gamma+1)\cdot \gamma}.
				\label{cancellation}
		\end{align}
		Because of this cancellation, which is the key to getting the form the form \eqref{eq:whichterms}, the only preceding orthogonal polynomials that contribute terms to $\phi_{j,k}$ are $\phi_{j,k-1}$ and $\phi_{j-1,k}$, so we have
		\begin{align*}
			\phi_{j,k}(z_1,z_2) &= z_1^jz_2^k - \frac{ \ipf{z_1^jz_2^k,\,\phi_{j,k-1}} }{\norm{\phi_{j,k-1}}_f^2}\phi_{j,k-1} - \frac{ \ipf{z_1^jz_2^k,\,\phi_{j-1,k}} }{\norm{\phi_{j-1,k}}_f^2}\phi_{j-1,k}\nonumber\\
				&= z_1^jz_2^k +\rtt \frac{j!k!}{(\gamma)_{j+k}}\brkt{\frac{1}{\norm{\phi_{j,k-1}}_f^2}\phi_{j,k-1} + \frac{1}{\norm{\phi_{j-1,k}}_f^2}\phi_{j-1,k}}.\nonumber\\
\end{align*}				
Using the inductive hypothesis about the norms and simplifying, we obtain
\begin{align}
				\phi_{j,k}(z_1,z_2)= z_1^jz_2^k+\rtt\frac{1}{\gamma+j+k}\brkt{k\phi_{j,k-1}+j\phi_{j-1,k}}. \label{eq:recursive}
		\end{align}
		This recursive formula can now be used to recover individual coefficients $\hat\phi_{j,k}(m,n)$ using the coefficients $\hat\phi_{j,k-1}(m,n)$ and $\hat\phi_{j-1,k}(m,n)$. We know that $\hat\phi_{j,k}(j,k)=1$, and in the case where $m=j$ (or, similarly, where $n=k$) we have $\hat\phi_{j-1,k}(j,n)=0$ (similarly, $\hat\phi_{j,k-1}(m,k)=0$). Let us first consider the case where $m=j$ and $n=0,1,\dots,k-1$, noting that the case where $n=k$ and $m=0,1,\dots,j-1$ proceeds similarly:
		\begin{align*}
			\hat\phi_{j,k}(j,n) 
			&= \rtt\frac{1}{\gamma+j+k} \brkt{k\hat\phi_{j,k-1}(j,n) +j \hat\phi_{j-1,k}(j,n) } \\
			&= \rtt\frac{1}{\gamma+j+k}\brkt{\frac{\sqrt{2}}{2}}^{j+k-1-j-n} 
					\frac{(\gamma+j+n)\cdots(\gamma+1)\cdot \gamma}{(\gamma+j+k-1)\cdots(\gamma+1)\cdot \gamma}  \\
						&\qquad\cdot\left( 
						k \brkt{\frac{j!(k-1)!}{j!n!} \frac{\brkt{j+k-1-j-n}!}{\brkt{j-j}!\brkt{k-1-n}!}} 						
						\right) \\
			&= \rtt^{k-n}\frac{(\gamma+j+n)\cdots(\gamma+1)\cdot \gamma}{(\gamma+j+k)\cdots(\gamma+1)\cdot \gamma}
						\cdot\frac{k!}{n!}\cdot 
						\frac{\brkt{k-1-n}!}{\brkt{k-1-n}!}	\\
			&= \rtt^{k-n}\frac{(\gamma+j+n)\cdots(\gamma+1)\cdot \gamma}{(\gamma+j+k)\cdots(\gamma+1)\cdot \gamma}
						\cdot\frac{k!}{n!}
		\end{align*}
		and this is what is obtained from substituting $m=j$ in \eqref{eq:closedcoeffs}.

	Now we consider the case where $n<k$ and $m<j$:
	\begin{align*}
			\hat\phi_{j,k}(m,n) 
			&= \rtt\frac{1}{\gamma+j+k} \brkt{k\hat\phi_{j,k-1}(m,n) +j \hat\phi_{j-1,k}(m,n) } \\~\\
			&= \rtt\frac{1}{\gamma+j+k}\brkt{\frac{\sqrt{2}}{2}}^{j+k-1-m-n} 
						\frac{(\gamma+m+n)\cdots(\gamma+1)\cdot \gamma}{(\gamma+j+k-1)\cdots(\gamma+1)\cdot \gamma}  \\
						&\qquad\cdot\left( 
						k \brkt{\frac{j!(k-1)!}{m!n!} \frac{\brkt{j+k-1-m-n}!}{\brkt{j-m}!\brkt{k-1-n}!}}\right.  \\
						&\left.\qquad\qquad+
						j\brkt{\frac{(j-1)!k!}{m!n!} \frac{\brkt{j-1+k-m-n}!}{\brkt{j-1-m}!\brkt{k-n}!}} 						
						\right) \\~\\
			&= \rtt^{j+k-m-n}\frac{(\gamma+m+n)\cdots(\gamma+1)\cdot \gamma}{(\gamma+j+k)\cdots(\gamma+1)\cdot \gamma}
						\cdot\frac{j!k!}{m!n!}  \\
						&\qquad\cdot\left( 
						\frac{\brkt{j+k-1-m-n}!(k-n)+\brkt{j-1+k-m-n}!(j-m)}{\brkt{j-m}!\brkt{k-n}!}						
						\right) \\~\\
			&= \brkt{\frac{\sqrt{2}}{2}}^{j+k-m-n} 
					\frac{(\gamma+m+n)\cdots(\gamma+1)\cdot \gamma}{(\gamma+j+k)\cdots(\gamma+1)\cdot \gamma} \cdot\frac{j!k!}{m!n!} \cdot
					\frac{\brkt{j+k-m-n}!}{\brkt{j-m}!\brkt{k-n}!},
		\end{align*}
		as needed.
		All that remains is to establish \eqref{eq:closednorms}. We use the recursive form \eqref{eq:recursive} and expand the inner product:
		\begin{align*}
			\ipf{\phi_{j,k},\,\phi_{j,k}}
			&= \ipf{z_1^jz_2^k,z_1^jz_2^k} + \rtt \frac{k}{\gamma+j+k}\ipf{z_1^jz_2^k,\phi_{j,k-1}} \\
			&\quad+  \rtt \frac{j}{\gamma+j+k}\ipf{z_1^jz_2^k,\phi_{j-1,k}} + \rtt \frac{k}{\gamma+j+k}\ipf{\phi_{j,k-1},z_1^jz_2^k}\\
			&\quad+ \rtt \frac{j}{\gamma+j+k}\ipf{\phi_{j-1,k},z_1^jz_2^k} +\half\frac{k^2}{(\gamma+j+k)^2}\normf{\phi_{j,k-1}}^2 \\
			&\quad+ \half\frac{kj}{(\gamma+j+k)^2}\ipf{\phi_{j,k-1},\phi_{j-1,k}} +\half\frac{jk}{(\gamma+j+k)^2}\ipf{\phi_{j,k-1},\phi_{j-1,k}}\\
			&\quad+ \half\frac{j^2}{(\gamma+j+k)^2}\normf{\phi_{j-1,k}}^2. 
\end{align*}		
Substituting the inductive values of the norms, \eqref{eq:iplessk}, \eqref{eq:iplessj}, and recalling that the $\phi$ are orthogonal, we obtain
\begin{align*}
			\ipf{\phi_{j,k},\,\phi_{j,k}}&= \frac{j!k!}{(\gamma+j+k-1)\cdots(\gamma+1)\cdot \gamma} \\&\quad+ \half\frac{(j+1)!k!}{(\gamma+j+k)\cdots(\gamma+1)\cdot \gamma}\\&\quad+ \half\frac{j!(k+1)!}{(\gamma+j+k)\cdots(\gamma+1)\cdot \gamma} \\
			&\quad+\sqrt{2} \frac{k}{\gamma+j+k}\brkt{-\rtt \frac{j!k!}{(\gamma+j+k-1)\cdots(\gamma+1)\cdot \gamma}} \\
			&\quad+\sqrt{2} \frac{j}{\gamma+j+k}\brkt{-\rtt \frac{j!k!}{(\gamma+j+k-1)\cdots(\gamma+1)\cdot \gamma}}\\
			&\quad+ \half\frac{k^2}{(\gamma+j+k)^2}\frac{\gamma+j+k}{(\gamma+j+k-1)}\frac{j!(k-1)!}{(\gamma+j+k-1-1)\cdots(\gamma+1)\cdot\gamma)}\\
			&\quad+ \half\frac{j^2}{(\gamma+j+k)^2}\frac{\gamma+j+k}{(\gamma+j+k-1)}\frac{(j-1)!k!}{(\gamma+j-1+k-1)\cdots(\gamma+1)\cdot\gamma)},
			\end{align*}
and simplifying yields
\begin{align*}
		\ipf{\phi_{j,k},\,\phi_{j,k}}	&= \frac{j!k!}{(\gamma+j+k-1)\cdots(\gamma+1)\cdot \gamma} +\frac{j!k!}{(\gamma+j+k)\cdots(\gamma+1)\cdot \gamma}\brkt{\frac{j+1}{2}+\frac{k+1}{2}}\\
			&\quad+ \frac{j!k!}{(\gamma+j+k)\cdots(\gamma+1)\cdot \gamma}\brkt{-k-j} +\frac{j!k!}{(\gamma+j+k)\cdots(\gamma+1)\cdot \gamma}\brkt{\frac{j}{2}+\frac{k}{2}}\\
			&= \frac{j!k!}{(\gamma+j+k)\cdots(\gamma+1)\cdot \gamma}\brkt{\gamma+j+k+\frac{j+1}{2}+\frac{k+1}{2}-j-k+\frac{j}{2}+\frac{k}{2}}\\
			&=\frac{j!k!}{(\gamma+j+k-1)\cdots(\gamma+1)\cdot \gamma}\cdot\frac{\gamma+j+k+1}{\gamma+j+k}.
		\end{align*}
	\end{proof}

	\begin{cor} \label{cor:recursive}
		The orthogonal polynomials given in Theorem \ref{thm:closed_form} can be written recursively as
		\begin{equation*}\label{eq:recursive}
			\phi_{j,k} = z^j w^k + \frac{\sqrt{2}}{2} \frac{1}{\gamma+j+k} \brkt{k\phi_{j,k-1} +j \phi_{j-1,k} }.
		\end{equation*}
	\end{cor}
\section{A family of optimal approximants}
Making use of the formula \eqref{OAvsOGformula}, we obtain information about optimal approximants to $1/(1-\frac{1}{\sqrt{2}}(z_1+z_2))$. We again set $\psi_{j,k}=\phi_{j,k}/\|\phi_{j,k}\|_{\gamma,f}$.
\begin{lemma}\label{lemma:OAcoeffs}
Let $\gamma>0$ be fixed. Then for $j,k\in \mathbb{N}$,
\[\langle 1, f\psi_{j,k}\rangle_{\gamma}\psi_{j,k}=\frac{\hat{\phi}_{j,k}(0,0)}{\|\phi_{j,k}\|^2}\phi_{j,k}=\left(\frac{\sqrt{2}}{2}\right)^{j+k}\frac{(j+k)!}{j!k!}\frac{\gamma}{\gamma+j+k+1}\phi_{j,k}.\]
\end{lemma}
\begin{proof}
From the power series expression for the norm in $\mathcal{H}_{\gamma}$, we have $\langle 1, f\psi_{j,k}\rangle_{\gamma}=\overline{(f\psi_{j,k})}(0)=\overline{\psi}_{j,k}(0,0)=\overline{\hat{\psi}_{j,k}(0,0)}$, and by definition, $\hat{\psi}_{j,k}(0,0)=\hat{\phi}_{j,k}(0,0)/\|\phi_{j,k}\|_{\gamma,f}$ which is real by \eqref{eq:closedcoeffs}.

It remains to compute
\[\phi_{j,k}(0,0)=\left(\frac{\sqrt{2}}{2}\right)^{j+k}\frac{\gamma}{(\gamma)_{j+k+1}}(j+k)!\]
and, simplifying, we obtain
\[\frac{\phi_{j,k}(0,0)}{\|\phi_{j,k}\|_{\gamma}^2}=\left(\frac{\sqrt{2}}{2}\right)^{j+k}\frac{(j+k)!}{j!k!}\frac{\gamma}{\gamma+j+k+1}.\]
\end{proof}
Setting $\Phi_{j,k}=\sum_{n=0}^j\sum_{m=0}^k\hat{\Phi}_{j,k}(m,n)z_1^mz_2^n$
where
\begin{multline}
\hat{\Phi}_{j,k}(m,n)\\=\left(\frac{\sqrt{2}}{2}\right)^{2(j+k)-m-n}\gamma \frac{(j+k)!}{m!n!}\frac{(\gamma)_{m+n+1}}{(\gamma)_{j+k+2}}\frac{(j+k-m-n)!}{(j-m)!(k-n)!}
\end{multline}
a representation formula for optimal approximants follows from Lemma \ref{lemma:OAcoeffs}:
\begin{thm}
For $\gamma>0$ fixed, we have
\[p_n^*(z_1,z_2)=\sum_{(j,k)\preceq (n_1,n_2)}\Phi_{j,k}(z_1,z_2)\]
where $(n_1,n_2)$ is the bidegree of the polynomial $p_n^*$.
\end{thm}
Explicitly, then,
\[p_0^*=\Phi_{0,0}, \quad p_1^*=\Phi_{0,0}+\Phi_{1,0}, \quad p_2^*=\Phi_{0,0}+\Phi_{1,0}+\Phi_{0,1},\]
\[p_3^*=\Phi_{0,0}+\Phi_{1,0}+\Phi_{0,1}+\Phi_{2,0}\quad p_4^*=\Phi_{0,0}+\Phi_{1,0}+\Phi_{0,1}+\Phi_{2,0}+\Phi_{1,1},\]
and so on. Some $p_n^*$'s for $\gamma=1$ (the Drury-Arveson space) are written out in \cite[Section 6.1]{MSAS1}. The first few optimal approximants for the Hardy space $H^2(\mathbb{B}^2)$ ($\gamma=2)$ are as follows:
\[p_0^*=\frac{2}{3} ,\quad p_1^*=\frac{3}{4}+\frac{1}{4}\sqrt{2}z_1,\quad p_2^*=\frac{5}{6}+\frac{\sqrt{2}}{4}(z_1+z_2),\quad p^*_3=\frac{17}{20}+\frac{3\sqrt{2}}{10}z_1+\frac{\sqrt{2}}{4}+\frac{1}{5}z_1^2,\]
\[p_4^*=\frac{53}{60}+\frac{7\sqrt{2}}{20}z_1+\frac{3\sqrt{2}}{10}+\frac{1}{5}z_1^2+\frac{2}{5}z_1z_2,\]
while the first optimal approximants in the Bergman space $A^2(\mathbb{B}^2)$ ($\gamma=3$) have the form
\[p_0^*=\frac{3}{4} ,\, p_1^*=\frac{33}{40}+\frac{3\sqrt{2}}{10}z_1,\, p_2^*=\frac{9}{10}+\frac{3\sqrt{2}}{10}(z_1+z_2),\, p^*_3=\frac{73}{80}+\frac{7\sqrt{2}}{20}z_1+\frac{3\sqrt{2}}{10}z_2+\frac{1}{4}z_1^2, \]
\[p_4^*=\frac{15}{16}+\frac{2\sqrt{2}}{5}z_1+\frac{7\sqrt{2}}{20}z_2+\frac{1}{4}z_1^2+\frac{1}{2}z_1z_2.\]
The symmetric form of $p_2^*$ above is explained in \cite[Section 6]{MSAS1}.

\section{An application}
Our results can be applied to study the cyclicity properties of the function $f=1-\frac{1}{\sqrt{2}}(z_1+z_2)$. Recall that $f$ is said to be {\it cyclic} in $\mathcal{H}_{\gamma}$ if the closure of the invariant subspace $\mathrm{span}\{z_1^jz_2^kf\colon j,k \in \mathbb{N} \}$ equals $\mathcal{H}_{\gamma}$.

Define the {\it optimal distance} $\nu^2_{n}(f,\mathcal{H}_{\gamma})=\|p_n^*f-1\|_{\mathcal{H}_{\gamma}}^2$: then $f$ is cyclic if and only if $\nu_{n}(f,\mathcal{H}_{\gamma}) \to 0$ as $n\to \infty$. 
Combining \cite[Corollary 5.3]{JLMS16} with our explicit formulas, we obtain the following.
\begin{corollary}
We have
\[\nu^2_{n}(f,\mathcal{H}_{\gamma})=1-\gamma^2\sum_{(j,k)\prec (n_1,n_2)}2^{-(j+k)}\frac{(j+k)!}{(\gamma)_{j+k+2}}\left(\begin{array}{c}j+k\\j\end{array}\right),\]
where $(n_1,n_2)$ is the bidegree of $p_n^*$. 
\end{corollary}
The function $f$ was already known to be cyclic in all $\mathcal{H}_{\gamma}$, but the above gives a precise description of how quickly the finite-dimensional subspaces $f\cdot \mathcal{P}_n$ fill up $\mathcal{H}_{\gamma}$. (The trick used to prove \cite[Proposition 23]{MSAS1} combined with \cite[Proposition 3.10]{FMS14} applied to the weight sequence 
$\omega(k)=k!/(\gamma)_k\asymp k^{\gamma-1}$ shows that the optimal distances have power law decay with exponent $-\gamma$.)

\section{Closing remarks}
As was highlighted in the course of the proof of Theorem \ref{thm:closed_form}, the cancellation in \eqref{cancellation} simplifies the structure of the orthogonal polynomials, giving rise to a relatively simple recursive relation that in turn allows us to write down an explicit formula for their coefficients; this phenomenon of course reflects the fact that the target function $f=1-\frac{1}{\sqrt{2}}(z_1+z_2)$ is well-adapted to the structure of $\mathcal{H}_{\gamma}$ (viz. also \cite[Proposition 23]{MSAS1}). 

In \cite{MSAS1}, optimal approximants to $1/f$ for the similar function $f=1-\frac{1}{2}(z_1+z_2)$ were examined for the family of Dirichlet-type spaces $\mathfrak{D}_{\alpha}$ on the bidisk 
$\mathbb{D}^2=\{(z_1,z_2)\in \mathbb{C}^2\colon |z_1|<1, |z_2|<1\}$, as were the corresponding orthogonal polynomials. While an analog of Lemma \ref{lem:winner_monoms} holds in that setting, cancellation fails for the orthogonal polynomials. Indeed, as is pointed out in \cite[Section 6]{MSAS1}, coefficients appearing in the orthogonal polynomials and optimal approximants in $\mathfrak{D}_{\alpha}$ in the bidisk exhibit sign changes and other complications, suggesting that obtaining a closed formula as well as precise estimates on optimal distances might be a harder problem than for the ball.

Returning to $\mathbb{B}^2$, we note that an analog of Lemma \ref{lem:winner_monoms} for the target function $g=\left(1-\frac{1}{\sqrt{2}}(z_1+z_2)\right)^2$, and indeed for other powers of $f$, is readily obtained. One can then proceed as we have done here in order to analyze orthogonal polynomials associated with the weight $g$. The computations quickly become more involved, but in principle one could attempt to obtain a recursive relation analogous to that in Corollary \ref{cor:recursive}, and then extract a closed formula for coefficients of orthogonal polynomials. As a sample, we invite the reader to verify that for $\gamma=1$ (the Drury-Arveson space), the orthogonal polynomials for the weight $g=f^2$ satisfy the relation
\begin{multline}
\phi_{j,k}=z_1^jz_2^k+\frac{\sqrt{2}}{j+k+2}(k\phi_{j,k-1}+j\phi_{j-1,k})\\-\frac{1}{(j+k+1)(j+k+2)}\left(\frac{k(k-1)}{2}\phi_{j,k-2}+jk\phi_{j-1,k-1}+\frac{j(j-1)}{2}\phi_{j-2,k}\right).
\end{multline}

\bibliography{bidisk_OG_OA.bib} 

\begin{thebibliography}{10}

\bibitem{JAM15}
{\sc B\'{e}n\'{e}teau, C., Condori, A.~A., Liaw, C., Seco, D., and Sola, A.~A.}
\newblock Cyclicity in {Dirichlet}-type spaces and extremal polynomials.
\newblock {\em J. Anal. Math. 126\/} (2015), 259--286.

\bibitem{JentZeros19}
{\sc B\'{e}n\'{e}teau, C., Khavinson, D., Liaw, C., Seco, D., and Simanek, B.}
\newblock Zeros of optimal polynomial approximants: {Jacobi} matrices and
  {Jentzsch}-type theorems.
\newblock {\em Rev. Mat. Iberoam 35}, 2 (2019), 607--642.

\bibitem{JLMS16}
{\sc B\'{e}n\'{e}teau, C., Khavinson, D., Liaw, C., Seco, D., and Sola, A.~A.}
\newblock Orthogonal polynomials, reproducing kernels, and zeros of optimal
  approximants.
\newblock {\em J. London Math. Soc. 94}, 3 (2016), 726--746.

\bibitem{BetalPrep2}
{\sc B\'{e}n\'{e}teau, C., Manolaki, M., and Seco, D.}
\newblock Boundary behavior of optimal polynomial approximants.
\newblock {\em Constr. Approx\/} (to appear).

\bibitem{CHZ18}
{\sc Cao, G., He, L., and Zhu, K.}
\newblock Spectral theory of multiplication operators on {H}ardy-{S}obolev
  spaces.
\newblock {\em J. Funct. Anal. 275\/} (2018), 1259--1279ß.

\bibitem{Chui80}
{\sc Chui, C.~K.}
\newblock Approximation by double least-squares inverses.
\newblock {\em J. Math. Anal. Appl. 75\/} (1980), 149--163.

\bibitem{CSW11}
{\sc Costea, {\c S}., Sawyer, E.~T., and Wick, B.~D.}
\newblock The corona theorem for the {D}rury-{A}rveson {H}ardy space and other
  holomorphic {B}esov-{S}obolev spaces on the unit ball in $\mathbb{C}^n$.
\newblock {\em Anal. PDE. 4\/} (2011), 499--550.

\bibitem{FMS14}
{\sc Fricain, E., Mashreghi, J., and Seco, D.}
\newblock Cyclicity in reproducing kernel {Hilbert} spaces of analytic
  functions.
\newblock {\em Comput. Methods Funct. Theory 14}, 4 (2014), 665--680.

\bibitem{GeroJEMS14}
{\sc Geronimo, J.~S., and Iliev, P.}
\newblock {Fej\'{e}r}-{Riesz} factorizations and the structure of bivariate
  polynomials orthogonal on the bi-circle.
\newblock {\em J. Euro. Math. Soc. 16\/} (2014), 1849--1880.

\bibitem{GWsiam07}
{\sc Geronimo, J.~S., and Woerdeman, H.~J.}
\newblock Two variable orthogonal polynomials on the bicircle and structured
  matrices.
\newblock {\em SIAM J. Matrix Anal. Appl. 29\/} (2007), 796--825.

\bibitem{RicSun16}
{\sc Richter, S., and Sunkes, J.}
\newblock Hankel operators, invariant subspaces, and cyclic vectors in the
  {Drury}-{Arveson} space.
\newblock {\em Proc. Amer. Math. Soc. 144\/} (2016), 2575--2586.

\bibitem{MSAS1}
{\sc Sargent, M., and Sola, A.~A.}
\newblock Optimal approximants and orthogonal polynomials in several variables.
\newblock {\em Preprint\/} (2020).

\bibitem{SecoSurvey}
{\sc Seco, D.}
\newblock Some problems on optimal approximants.
\newblock In {\em Recent progress on operator theory and approximation in
  spaces of analytic functions}, C.~B\'{e}n\'{e}teau, A.~A. Condori, C.~Liaw,
  W.~T. Ross, and A.~A. Sola, Eds. Amer. Math. Soc., Providence, RI, 2016,
  pp.~193--205.

\bibitem{Zhu}
{\sc Zhu, K.}
\newblock {\em Spaces of holomorphic functions in the unit ball}.
\newblock Springer-Verlag, 2005.

\end{thebibliography}
\bibliographystyle{acm}

\end{document}